\documentclass{amsart}

\usepackage[all, cmtip]{xy}
\usepackage{amssymb,amsmath,amscd,xy,graphicx,textcomp,mathtools}
\usepackage{epsf} 
\usepackage{hyperref}
\usepackage{booktabs}

\hypersetup{
  colorlinks   = true, 
  urlcolor     = blue, 
  linkcolor    = blue, 
  citecolor   = blue 
}

\newtheorem{theorem}{Theorem}[section]
\newtheorem{lemma}[theorem]{Lemma}
\newtheorem{corollary}[theorem]{Corollary}
\newtheorem{definition}[theorem]{Definition}

\newtheorem{remark}[theorem]{\it Remark}
\newtheorem{example}[theorem]{Example}
\newtheorem{proposition}[theorem]{Proposition}
\newtheorem{conjecture}[theorem]{Conjecture}

\def\GL{\mathrm{GL}}

\def\SL{\mathrm{SL}}

\def\C{\mathbb{C}}
\def\R{\mathbb{R}}
\def\Z{\mathbb{Z}}

\def\Z{\mathbb{Z}}

\def\tree{\mathcal{T}}

\def\q{/\!\!/}
\def\ql{\backslash \!\! \backslash}

\def\X{\mathcal{X}}

\title[Character Varieties are Gorenstein]{Character Varieties of Free Groups are Gorenstein but not always Factorial}

\author[S. Lawton]{Sean Lawton}

\address{Department of Mathematical Sciences, George Mason University,
4400 University Drive,
Fairfax, Virginia  22030, USA}

\email{slawton3@gmu.edu}

\author[C. Manon]{Christopher Manon}
\address{Department of Mathematical Sciences, George Mason University,
4400 University Drive,
Fairfax, Virginia  22030, USA}

\email{cmanon@gmu.edu}



\begin{document}

\begin{abstract}
Fix a rank $g$ free group $F_g$ and a connected reductive complex algebraic group $G$.  Let $\mathcal{X}(F_g, G)$ be the $G$--character variety of $F_g$.  When the derived subgroup $DG<G$ is simply connected we show that $\mathcal{X}(F_g, G)$ is factorial (which implies it is Gorenstein), and provide examples to show that when $DG$ is not simply connected $\mathcal{X}(F_g, G)$ need not even be locally factorial.  Despite the general failure of factoriality of these moduli spaces, using different methods, we show that $\mathcal{X}(F_g, G)$ is always Gorenstein.  
\end{abstract} 


\maketitle

\section{Introduction}

Let $\Gamma$ be a finitely generated group (generated by $g$ elements), perhaps the fundamental group of a manifold or orbifold $M$, and let $G$ be a reductive affine algebraic group over an algebraically closed field $\Bbbk$.  Then the collection of group homomorphisms $\mathrm{Hom}(\Gamma, G)$ is naturally a $\Bbbk$--variety cut out of the product variety $G^g$ via the relations defining $\Gamma$.  Since $G$ admits a faithful morphism into $\mathrm{GL}_n(\Bbbk)$, we call $\mathrm{Hom}(\Gamma, G)$ the $G$--representation variety of $\Gamma$.  Standard in representation theory is that two representations are equivalent if they are conjugate.  As we are interested in $G$--valued representations, we consider the conjugation action of $G$ on the representation variety.  The orbit space $\mathrm{Hom}(\Gamma, G)/G$ is generally not Hausdorff, but it is homotopic (see \cite[Proposition 3.4]{FLR}) to the geometric points of the Geometric Invariant Theory (GIT) quotient $\mathcal{X}(\Gamma, G):=\mathrm{Hom}(\Gamma, G)\q G$; which as usual, is the spectrum of the ring of $G$--invariant elements in the coordinate ring $\Bbbk[\mathrm{Hom}(\Gamma, G)]$.  

The spaces $\mathcal{X}(\Gamma, G)$, the {\it $G$-character varieties of $\Gamma$}, constitute a large class of affine varieties (see \cite{Rap}).  More importantly, they are of central interest in the study of moduli spaces (see \cite{A-B, Hitchin, NS, Simpson1, Simpson2}); finding applications in mathematical physics (see \cite{Kac-Smilga,Witten,HaTh,KW}), the study of geometric manifolds (see \cite{GoldmanGStructure}), and knot theory (see \cite{CCGLS}). 

The purpose of this note is to prove that the singularities of a character variety $\mathcal{X}(F_g, G)$ are always well behaved, when $G$ is reductive and $F_g$ is a rank $g$ free group.  This is in stark contrast to the situation when $\Gamma$ is not free; 
see \cite{KaMi}. In particular, we characterize when $\mathcal{X}(F_g, G)$ is factorial (in arbitrary characteristic) and show it is always Gorenstein (when $\Bbbk=\C$).  This allows us to describe conditions when the Picard group of $\mathcal{X}(F_g,G)$ is trivial. 

In general, regular rings (smooth varieties) are local complete-intersection rings (varieties locally cut out by codimension relations) which are Gorenstein.  All Gorenstein rings are Cohen-Macaulay, and assuming the singular locus lies in codimension at least 2, this further implies normality (which implies irreducibility). So Gorenstein varieties (coordinate rings are Gorenstein) are situated between complete-intersections and Cohen-Macaulay varieties.  Factorial varieties (coordinate rings are unique factorization domains) fall slightly in between.  When they are additionally normal, they have the property that all irreducible hypersurfaces (divisors) come from rational functions (principal divisors).  A factorial Cohen-Macaulay variety is automatically Gorenstein (but not conversely).  For $G=\SL_n(\C)$, we describe exactly when these properties hold for $\X(F_g,G)$; see Example \ref{typeA}.

Here is the main theorem of the paper:
\begin{theorem}\label{maintheorem}
For any $g\geqslant 1$ and any connected complex reductive algebraic group $G$, the space $\mathcal{X}(F_g,G)$ is Gorenstein.  If the derived subgroup $DG<G$ is simply connected, then $\mathcal{X}(F_g,G)$ is factorial.   If the derived subgroup $DG<G$ is not simply connected, there are examples where $\mathcal{X}(F_g,G)$ is not factorial, nor locally factorial.  
\end{theorem}

We prove the factoriality parts of the above theorem (in arbitrary characteristic) in Section \ref{ufd-sec}.  As an interesting corollary, we give conditions for the Picard group of the character variety to be trivial: $\mathrm{Pic}(\mathcal{X}(F_g,G))=0$ if $DG$ is simply connected.   The proof of the Gorenstein part of Theorem \ref{maintheorem} makes up the rest of the paper.  As the proof uses the representation theory of algebraic groups, in Section 
\ref{gor-sec} we review the relevant ideas and build up the machinery needed for the proof which is in Sections \ref{application-sec} and \ref{gor-proof-sec}.

In the course of our proof of Theorem \ref{maintheorem} we also obtain similar results for the configuration space $P_n(G)$ of $n$ points in the homogeneous variety $G / U$ for $U \subset G$ a maximal unipotent subgroup; see Theorem \ref{pnufdgor}, and Remark \ref{pngor}.  These spaces play a central role in the study of saturation problems which arise from the theory of tensor product invariants in the representation theory of $G$,  see e.g. Section $9$ of \cite{KLM}.

\subsection*{Acknowledgments} 
We thank both Brian Conrad and Vladimir Popov for giving references and comments concerning Lemma \ref{liegrouplemma}; this was extremely helpful.  We also thank Neil Epstein for several useful discussions on factorial rings.  Lawton was partially supported by grants from the Simons Foundation (Collaboration \#245642) and the National Science Foundation (DMS \#1309376).  Additionally, Lawton acknowledges support by the National Science Foundation under Grant No. 0932078000 while in residence at the Mathematical Sciences Research Institute in Berkeley, California, during the Spring 2015 semester.  Finally, we thank an anonymous referee for helpful references.

\section{Factoriality}\label{ufd-sec}

Let $X$ be a normal affine variety over an algebraically closed field $\Bbbk$ of arbitrary characteristic. Then \cite[Proposition 6.2]{hartshorne} shows that the coordinate ring $\Bbbk[X]$ is a unique factorization domain (UFD) if and only if the divisor class group $\mathrm{Cl}(X)=0$. From \cite[Theorems 11.8 and 11.10]{Eisenbud}, we see that the Picard group $\mathrm{Pic}(X)$ injects into $\mathrm{Cl}(X)$, and this injection is an isomorphism if and only if $X$ is locally factorial.

In this generality, a connected semisimple algebraic group $G$ is {\it simply connected} if it does not admit any non-trivial isogeny. The following lemma is implied by the work of Popov \cite{popov}, and is essential to the proof of Theorem \ref{ufd-theorem} below.  

\begin{lemma}\label{liegrouplemma}
Let $G$ be an affine algebraic $\Bbbk$--group, where $\Bbbk$ is an algebraically closed field of arbitrary characteristic.  Let $R_u$ be the unipotent radical of $G$.  Then $G$ is factorial if and only if $G$ is connected and the derived subgroup $[G/R_u,G/R_u]$ is simply connected. 
\end{lemma}

\begin{proof}
Since $G$ is smooth, it suffices to compute the Picard group of $G$. If $G$ is not connected $\Bbbk[G]$ is not a domain, so we now assume $G$ is connected.

Now, the natural sequence $1 \to R_u \to G \to G/R_u \to 1$ gives a sequence $\mathrm{X}(R_u) \to\mathrm{Pic}(G/R_u)\to\mathrm{Pic}(G) \to \mathrm{Pic}(R_u)\to 0$ by \cite[Remark 6.11.3]{sansuc}, where $\mathrm{X}(R_u)$ is the character group of the unipotent radical $R_u$.  However, the character group $\mathrm{X}(R_u)$ is trivial since $R_u$ is isomorphic to an affine space by \cite{Grothendieck1958}.
For the same reason $\mathrm{Pic}(R_u)=0$ by \cite[Proposition 6.6]{hartshorne}.  Thus, $\mathrm{Pic}(G/R_u)=\mathrm{Pic}(G)$, where $G/R_u$ is reductive (since its unipotent radical is trivial).  Accordingly, we now assume $G$ is reductive and consider the exact sequence $1 \to DG \to G \to G/DG \to 1$, where $DG=[G,G]$ is the derived subgroup of $G$.  

By the Central Isogeny Theorem, $G$ is isomorphic to $DG \times_F Z$ where $Z$ is the radical of $G$ (the center, a torus) and $F=Z\cap DG$ acts by $f(d,z)=(df,f^{-1}z)$.  Therefore, $G/DG\cong Z/F$.   However, since $G$ is a split reductive group (since $\Bbbk$ is separably closed), $Z$ is a standard torus $(\Bbbk^*)^m$ and therefore $T:=Z/F$ is isomorphic to a standard torus too since quotients of split tori are likewise split tori.

Again by \cite[Remark 6.11.3]{sansuc} we have an exact sequence $\mathrm{Pic}(T)\to\mathrm{Pic}(G)\to \mathrm{Pic}(DG)\to 0$.  However, $\Bbbk[T]$ is a multivariable Laurent ring.  As such it arises via repeated localizations starting from a polynomial ring.  Therefore, $T$ is factorial and $\mathrm{Cl}(T)=\mathrm{Pic}(T)=0$.  We conclude that $\mathrm{Pic}(G)\cong \mathrm{Pic}(DG)$.

Putting these observations together with \cite[Proposition 1]{popov}, 
we conclude that for any affine algebraic $\Bbbk$--group, $\mathrm{Pic}(G)=0$ if and only if $G$ is connected and the derived subgroup $[G/R_u,G/R_u]$ is simply connected.  More generally, from \cite[Corollary, Theorem 6]{popov} or \cite[Proposition 3.6]{Iversen-76}, we have $\mathrm{Pic}(G)=\pi_1([G/R_u,G/R_u])$. 
\end{proof}

\begin{remark}{\rm
One must take care reducing to the semisimple case in Lemma \ref{liegrouplemma}.  For example, letting $G=\GL_n(\C)$, if one considers the quotient by the full radical the semisimple group $\mathrm{P}\GL_n(\C)$ is obtained, which is not simply connected although $G$ is factorial.  However, as expected given the above lemma, its derived subgroup $DG=\SL_n(\C)$ is simply connected.
}\end{remark}

\begin{theorem}\label{ufd-theorem}
Let $G$ be a connected reductive affine algebraic $\Bbbk$--group, where $\Bbbk$ is an algebraically closed field of arbitrary characteristic, and let $F_g$ a rank $g$ free group.  Then $\mathcal{X}(F_g, G)$ is factorial if $DG$ is simply connected.
\end{theorem}

\begin{proof}
The character variety $\mathcal{X}(F_g,G)$ is the GIT quotient of $\mathrm{Hom}(F_g,G)\cong G^g$ by the conjugation action of $PG:=G/Z(G)$.  The character group $\mathrm{Hom}(PG,\Bbbk^*)$ is trivial since $PG$ is semisimple, and so the ring of invariants $\Bbbk[G^g]^{PG}$ is a UFD by \cite{Pop72}[Remark 3, Page 376], since $G^g$ is factorial by the previous Lemma \ref{liegrouplemma}.  We note that to apply Lemma \ref{liegrouplemma} observe that $G^g$ is a reductive group whose derived subgroup is $DG^g,$ and the latter is simply connected since $DG$ is simply connected by assumption.
\end{proof}

In \cite{KuNa} the Picard group of the moduli space of semistable principal $G$-bundles over a closed surface of genus at least 2 is shown to be $\Z$, when $G$ is a simply connected simple complex algebraic group.  They also show this moduli space is Gorenstein.  The next two corollaries may be considered analogues to these results.

\begin{corollary}
Let $G$ be as in Theorem \ref{ufd-theorem}. Then
$\mathrm{Pic}(\mathcal{X}(F_g,G))=0$ if $DG$ is simply connected.
\end{corollary}

\begin{proof}
In Theorem \ref{ufd-theorem}, we see that the coordinate ring $\Bbbk[\mathcal{X}(F_g,G)]$ is a UFD and thus its class group is trivial.  However, $\mathrm{Pic}(\mathcal{X}(F_g,G))$ injects into the class group, and so is likewise trivial.
\end{proof}

\begin{corollary}
Let $G$ be as in Theorem \ref{ufd-theorem}. Then
$\mathcal{X}(F_g,G)$ is Gorenstein if $DG$ is simply connected.
\end{corollary}

\begin{proof}
By a theorem of Murthy \cite{Murthy}, it suffices to show that the coordinate ring $\Bbbk[\mathcal{X}(F_g, G)]$ is Cohen-Macaulay and a UFD.  The former is a consequence of a theorem of Hochster and Roberts \cite{HR}, and the latter was shown above in Theorem \ref{ufd-theorem}.
\end{proof}

Unfortunately, when $DG$ is not simply connected, $\mathcal{X}(F_g, G)$ need not be even locally factorial over $\C$, as the next example shows.  Nevertheless, in the sequel we show that $\mathcal{X}(F_g, G)$ is Gorenstein in general (over $\C$).

\begin{example}{\rm
By \cite[Corollary 6]{Sikora-SO}, the irreducible coordinate ring of $\mathcal{X}(F_2,\mathrm{PSL}_2(\C))$ is isomorphic to $\C[x,y,z,w]/\langle xyz-w^2\rangle$.  Therefore, it is certainly not a UFD.  However, this example shows that character varieties are not generally even locally factorial.  A normal affine variety $V$ is local factorial if and only if $\mathrm{Pic}(V)=\mathrm{Cl}(V)$, by \cite[Theorems 11.8 and 11.10]{Eisenbud}.  However, $\mathcal{X}(F_2,\mathrm{PSL}_2(\C))$ is normal since it is the GIT quotient of a smooth $($and hence normal$)$ variety by a reductive group.   Since $\mathcal{X}(F_2,\mathrm{PSL}_2(\C))$ is not factorial, we conclude that its class group is non-trivial; but its Picard group is trivial as we now explain.  $\mathcal{X}(F_2,\mathrm{PSL}_2(\C))\cong \mathcal{X}(F_2,\mathrm{SL}_2(\C))\q (\mathbb{Z}/2\mathbb{Z})^2$, and $\mathcal{X}(F_2,\mathrm{SL}_2(\C))\cong \C^3$ by results of Vogt \cite{Vo}.  Since it is a finite quotient of an affine space, we conclude that its Picard group is trivial $($see \cite{Kang}$)$. Thus, $\mathcal{X}(F_2,\mathrm{PSL}_2(\C))$ is not locally factorial.  Along similar lines, Example 7 in \cite{Sikora-SO} also shows that $\mathcal{X}(F_2,\mathrm{SO}_4(\C))$ is not locally factorial.}
\end{example}

\begin{example}\label{typeA}{\rm
Since $\X(F_g,\SL_n(\C))$ is the GIT quotient of a smooth variety by a reductive group, the coordinate ring $\C[\X(F_g,\SL_n(\C))]$ is integrally closed, Cohen-Macaulay, and has rational singularities.  By \cite[Theorem 1.1]{FlLa3} $\C[\X(F_g,\SL_n(\C))]$ is regular if and only if $(g,n)=(1,n),(g,1),$ or $(2,2)$, and as mentioned in \cite[Remark 4.6]{FlLa3}, \cite{LT} implies that $\C[\X(F_g,\SL_n(\C))]$ is a complete intersection if and only if $(g,n)=(1,n),(g,1),(2,2),(2,3),$ or $(3,2)$.  In these cases the isomorphism type of the rings is also known; see \cite[Section 2.2]{FlLa3}.  We have now further established that $\C[\X(F_g,\SL_n(\C))]$ has trivial Picard group, is a UFD, and hence is Gorenstein.}

\end{example}

\begin{conjecture}
With the preceding examples in mind, we conjecture that $\mathcal{X}(F_g, G)$ is factorial if and only if $DG$ is simply connected.
\end{conjecture}

\begin{remark}{\rm
We have established half of this conjecture here, and showed that $\mathrm{Hom}(F_g,G)$ is factorial if and only if $DG$ is simply connected.  It remains to prove that the GIT quotient of  $\mathrm{Hom}(F_g,G)$ cannot be factorial since  $\mathrm{Hom}(F_g,G)$ itself is not. }
\end{remark}

\section{Gorenstein graded algebras and semigroup algebras}\label{gor-sec}

For the remainder of the paper we work over $\C$.  A Noetherian ring $R$ is said to be Gorenstein if every localization $R_m$ at a maximal ideal $m \subset R$ has finite injective dimension as a module over itself: $inj_{R_m}(R_m) < \infty$ (see \cite[Definition 3.1.18]{BH}). Gorenstein algebras are also Cohen-Macaulay, and in all the cases we consider here, they can be characterized in the class of Cohen-Macaulay rings by the behavior of an associated object called the canonical module $\Omega(R).$  When $R$ is a homomorphic image of a polynomial ring $S = \C[x_1, \ldots, x_n]$ (this is the case for all algebras we study in this paper), one can define the canonical module as $Ext^{n-dim(R)}_S(R, S)$. This module agrees with the module of top differential forms of $R$ on the non-singular locus of $Spec(R)$; the $\C-$algebra $R$ is then Gorenstein when $\Omega(R)$ is a principal module isomorphic to $R$ (see \cite[Theorem 3.3.7]{BH}). The module $\Omega(R)$ was introduced by Grothendieck to study duality theory for singular varieties, in particular when $R$ is Gorenstein the Betti numbers of $R$ as an $S-$module satisfy $\beta^S_i(R) = \beta^S_{n-dim(R)-i}(R)$.  An illustrative example of a Gorenstein algebra is $S$ itself, where $\Omega(S)$ is generated by the top form $dx_1 \wedge \ldots \wedge dx_n$.

It is a recurring theme in commutative algebra that algebraic properties are often easier to establish for rings with gradings, this is also the case with the Gorenstein property.   In the presence of a positive multigrading, a theorem of Stanley (see \ref{stanley} below) establishes that the Gorenstein property is characterized by features of the Hilbert function, when the ring in question is a Cohen-Macaulay Noetherian domain.  Normal affine semigroup algebras provide an extreme special case of multigraded algebras, in this case the Gorenstein property can be proved using the combinatorics of the associated semigroup.  In this section we will recall these two results, and indicate how degeneration techniques allow us to use these powerful methods on more general classes of algebras which may not have a grading, in particular the coordinate rings of character varieties.

\subsection{Gorenstein graded algebras}

 Let $R = \bigoplus_{\lambda \in \Z_{\geqslant 0}^n} R_{\lambda}$ be a multigraded Noetherian domain with $\C = R_0.$  Let $\Z[[\mathbf{t}]]$ be the ring of formal power series in the monomials $\mathbf{t}^{\lambda} = t_1^{\lambda_1}\cdots t_n^{\lambda_n}$, $\lambda \in \Z_{\geqslant 0}^n$.  Recall that the Hilbert series of $R$ is the following element of $\Z[[\mathbf{t}]]$:

\begin{equation}
H(R, \mathbf{t}) = \sum_{\lambda \in \Z_{\geqslant 0}^n} dim(R_{\lambda})\mathbf{t}^{\lambda}
\end{equation}

\noindent
 With these hypotheses, the series $H(R, \mathbf{t})$ is always the expansion of a rational function (see \cite[Theorem 2.3]{stanleybook}):

\begin{equation}
H(R, \mathbf{t}) = \frac{\mathbf{t}^{\beta}P(\mathbf{t})}{\prod (1-\mathbf{t}^{\alpha_i})}
\end{equation}

The weights $\alpha_i$ appearing in the denominator of this function can be taken to be the multidegrees of a set of homogeneous generators of $R.$   The following theorem of Stanley (see \cite[Theorem 12.7]{stanleybook} and \cite[Theorem 6.1]{stanley}) gives a characterization of the Gorenstein property in terms of $H(R, \mathbf{t})$.

\begin{theorem}[Stanley]\label{stanley}
Let $R$ be a $\Z_{\geqslant 0}^n$--multigraded Cohen-Macaulay Noetherian domain, then $R$ is Gorenstein if and only if the following holds for some $\omega \in \Z^n$:

\begin{equation}
H(R, \mathbf{t}^{-1}) = (-1)^d\mathbf{t}^{\omega}H(R, \mathbf{t}).
\end{equation}

Furthermore, the weight $\omega$ is the multidegree of the generator of the unique $\Z^n-$ multigraded canonical module of $R.$

\end{theorem}

A matrix $\delta \in \GL_n(\Z)$ acts on $\Z[[\mathbf{t}]]$ by permuting the monomials $\mathbf{t}^{\alpha}$, this action is compatible with rational expansions. 
Let $\Delta \subset \GL_n(\Z)$ by a finite group of such matrices which fix the Hilbert function $H(R, \mathbf{t})$.  Suppose $R$ has a homogeneous generating set of degrees $A = \{\alpha_1, \ldots, \alpha_k\}$, then letting $\Delta\circ A = \{ \alpha = \delta \circ \alpha_i | \alpha_i \in A, \delta \in \Delta\}$, any homogeneous basis of $\bigoplus_{\alpha \in \Delta \circ A} R_{\alpha}$ likewise generates $R$.  
As this set is still finite, it follows that we may choose a rational expression for $H(R, \mathbf{t})$ with denominator fixed by the action of $\Delta$.  The rational expressions for  $H(R, \mathbf{t})$ and $H(R, \mathbf{t}^{-1})$ are therefore left unchanged by the action of any $\delta \in \Delta$. The following proposition is a consequence of this observation. 

\begin{proposition}\label{symmetricgenerator}
Let $R$ be a $\Z_{\geqslant 0}^n$--graded Gorenstein Noetherian domain, and let  $\Delta$ be a finite group of $\Z-$linear automorphisms which fix $H(R, \mathbf{t}),$ then the multidegree $\omega$ of the generator of the canonical module of $R$ is preserved by each $\delta \in \Delta$.
\end{proposition}

\subsection{Gorenstein semigroup algebras}\label{gor-semi-alg}

Let  $L \subset \R^n$ be a lattice, and let $\mathcal{P} \subset \R^n$ be an integral polyhedral cone.  Recall that this means that $\mathcal{P}$ is obtained as a set of solutions to a finite collection of $L$--integral linear inequalities: $\langle v, a_i\rangle \geqslant 0$, $a_i \in L$.   The set $P = \mathcal{P} \cap L$ is naturally a semigroup under addition with identity $0$, we let $\C[P]$ be the associated semigroup algebra. The semigroup $P$ is said to be positive if $0$ is the only invertible element.     For any facet $F \subset \mathcal{P}$, let $P_F \subset P$ be the points not in $F$.  The vector space $\C[P_F]$ is a prime ideal in $\C[P]$, likewise the set of points $int(P) = \cap_{F} P_F$ spans an ideal  $\C[int(P)] \subset \C[P].$ For the following  see \cite[Theorem 6.35]{BH}. 

\begin{theorem}
The algebra $\C[P]$ is a Cohen-Macaulay ring, and it has a canonical module isomorphic to the ideal $\C[int(P)]$.  Furthermore, 
if $P$ is positive, $\C[int(P)]$ provides the unique $P-$grading on the canonical module of $\C[P]$. 
\end{theorem}

It follows that $\C[P]$ is Gorenstein if and only if $\C[int(P)]$ is a principal module, indeed the following is Corollary $6.3.8$ in \cite{BH}.

\begin{proposition}
The affine semigroup algebra $\C[P]$ is Gorenstein if and only if there is a $w \in int(P)$ such that any $x \in int(P)$ can be expressed as $x = y + w$ for some $y \in P.$
\end{proposition}

\begin{example}{\rm
If we take $P$ to be the non-negative orthant $\R_{\geqslant 0}^n \subset \R^n$, then $\C[P]$
is isomorphic to the polynomial ring $S = \C[x_1, \ldots, x_n]$.  The multiweight
of the generator $dx_1\wedge \ldots \wedge dx_n$ of $\Omega(S)$ is $(1, \ldots, 1)$,
and indeed this weight generates $int(P)$.}
\end{example}

\begin{corollary}\label{gorrestrict}
Let $\mathcal{P} \subset \R^n$ be an integral polyhedral cone with respect to a lattice $L \subset \R^n$, and let $L \subset L'$ be an inclusion of lattices.  Then if $\C[\mathcal{P} \cap L']$ is Gorenstein and the generator $w \in int(\mathcal{P} \cap L')$ is in $L$, then $\C[\mathcal{P} \cap L]$ is Gorenstein.  
\end{corollary}

\subsection{Deforming the Gorenstein property}

 In what follows a flat degeneration of a variety $X$ over an algebraically closed field $\C$ to a variety $X(\Delta)$ is taken to mean a scheme $E$ with a flat map $\pi: E \to \mathbb{A}^1_\C$, such that $\pi^{-1}(0) = X(\Delta)$ and $\pi^{-1}(C) = X$ for $C \neq 0.$   When the special fiber $X(\Delta)$ of this family is a normal toric variety, we say that $\pi: E \to \C$ defines a toric degeneration of $X.$   We show that the Gorenstein property behaves well with respect to flat degenerations.  We take $A$ to be a commutative Noetherian domain over a field $\C$, and we fix an increasing $\Z_{\geqslant 0}$ algebra filtration $F$ of $A$ by $\C$--vector spaces with $F_0 = \C$: 

\begin{equation}
A = \bigcup_{m \in \Z_{\geqslant 0}} F_m.
\end{equation}

\noindent
This filtration defines two related algebras, the Rees algebra:

\begin{equation}
R_F(A) = \bigoplus_{m \in \Z_{\geqslant 0}} F_m,
\end{equation}

\noindent
and the associated graded algebra:

\begin{equation}
gr_F(A) = \bigoplus_{m \in \Z_{\geqslant 0}} F_m/F_{m-1}.
\end{equation}

Let $t \in F_1$ be the copy of the identity of $1$ of filtration level $1$.  For any
filtration space $F_m \subset A$, the subspace $tF_m \subset F_{m+1}$
is the copy of $F_m$ inside $F_{m+1},$ this defines a $\C[t]$--algebra structure on $R_F(A).$   Furthermore, since $t$ is non-invertible and a non-zerodivisor, $R_F(A)$ is a faithfully flat $\C[t]$--module.  Let $E = Spec(R_F(A))$, then $E$ is a flat
family of affine schemes over the affine line.  We have the following exact sequence of $R_F(A)$ modules: 

$$
\begin{CD}
0 @>>> tR_F(A) @>>> R_F(A) @>>> gr_F(A) @>>> 0.\\
\end{CD}
$$ 

In particular, the fiber $E_0$ over $0$ is $Spec(gr_F(A))$.  Identifying $t$ with $1$ results
in the union $\cup_{m \in \Z_{\geqslant 0}} F_m = A$, so we identify $E_1$ (indeed $E_C$ for any $C \neq 0$) with $Spec(A).$ In this way, $F$ defines a flat degeneration of $A$ to the algebra $gr_F(A)$.  All the degenerations we use in this paper are of this form.  

\begin{proposition}\label{deggor}
If $gr_F(A)$ is Gorenstein then $A$ and $R_F(A)$ are Gorenstein as well. 
\end{proposition}

\begin{proof}
We use Corollary $3.3.15$ of \cite{BH}: if $f: (R, m) \to (S, n)$ is a flat, local map
of Noetherian local rings, and if $R$ and $S/mS$ are Gorenstein, then $S$ is
Gorenstein as well.   For $(R, m)$ we take $(\C[t]_t, t)$ and for $(S, n)$ we take $(R_F(A)_M, M)$, where $M = \bigoplus_{m > 0} F_m$.   A graded algebra such as $R_F(A)$ with $F_0 = \C$ is Gorenstein if and only if its localization at its maximal graded ideal is Gorenstein (see e.g.  \cite[Exercise 3.6.20]{BH}), so $gr_F(A)$ Gorenstein implies that the localization $gr_F(A) \otimes_{R_F(A)}R_F(A)_M$
is Gorenstein.  By the above exact sequence, this is $R_F(A)_M/t$.  This proves that $R_F(A)$ is Gorenstein, and implies that  $R_F(A)/(t-C)$ is Gorenstein for any $C \in \C,$ since these algebras are all quotients of $R_F(A)$ by an irreducible element.   In particular, $A = R_F(A)/(t-1)$ must also be Gorenstein. 
\end{proof}

\section{The space \texorpdfstring{$P_3(G)$}{P3G}}\label{application-sec}

The proof of Theorem \ref{maintheorem} depends on establishing the Gorenstein property for the configuration space $P_3(G)$ (or rather a degeneration of this space). We start by showing this property holds when $G$ is semisimple and simply connected, and then we use combinatorial methods to extend the result to all reductive $G.$

 For $G$ a connected reductive group and $U \subset G$ a maximal unipotent subgroup, $P_n(G)$ is constructed as the following GIT quotient: 

\begin{equation}
P_n(G) = G \ql [G / U]^n.
\end{equation}

The variety $G / U$ is the spectrum of the total coordinate ring of the flag variety $G/B$, 
for $B$ the Borel subgroup which contains $U.$    The maximal torus $T$ normalizes $U$, so there is a residual $T$--action on the right hand side of $G / U$, and the isotypical components of the coordinate ring $\C[G / U]$ with respect to this action are precisely the irreducible representations of $G$. In particular the support of the corresponding multigrading of this decomposition is the set of dominant weights in the Weyl chamber $\Delta$, see \cite[Section 2]{popovvinberg72}:

\begin{equation}
\C[G / U] = \bigoplus_{\lambda \in \Delta} V(\lambda).
\end{equation}

The $T$--actions on the components $G / U$ in the definition of $P_n(G)$ endow this space with a rational $T^n$--action, the isotypical components of this action are the $n-$fold tensor product invariant spaces of $G:$

\begin{equation}
\C[P_n(G)] = \bigoplus_{\vec{\lambda} \in \Delta^n} [V(\lambda_1) \otimes \ldots \otimes V(\lambda_n)]^G.
\end{equation}

\noindent
The support of the $\Delta^n$--multigrading of $\C[P_n(G)]$ is an affine semigroup $T_n(G) \subset \Delta^n$.  A tuple $\vec{\lambda}$ is in $T_n(G)$ if and only if there is $G-$invariant vector in the tensor product $V(\lambda_1) \otimes \ldots \otimes V(\lambda_n)$; note that this is also the support of the multigraded Hilbert function $H_{T^n}(P_n(G), \mathbf{t})$.  We let $\mathcal{T}_n(G)$ be the convex hull of $T_n(G).$

\subsection{The semisimple and simply connected case}

We begin with a proof that $P_n(G)$ is Gorenstein when $G$ is semisimple and simply connected.   We distinguish the lemma from the general result as it is a consequence of more general techniques from invariant theory.

\begin{lemma}\label{pnufdgor}
Let $G$ be semisimple and simply connected, then $P_n(G)$ is factorial and Gorenstein. 
\end{lemma}

\begin{proof}
We show that $\C[P_n(G)]$ is factorial and Cohen-Macaulay, then the lemma follows from the result of Murthy, \cite{Murthy}.   The algebra $\C[P_n(G)]$ is obtained as the invariant subring of $\C[G^n]$ with respect to an action by $G \times U^n$, and we may exchange the order by which we take invariants.  The invariants in $\C[G^n]$ by the left action of $G$ give the coordinate ring of $G^{n-1}$, which is factorial by Lemma \ref{liegrouplemma}.  By \cite[Theorem 3.17]{popovvinberg}, factoriality is preserved by passing to unipotent invariants, since $U^n$ is connected with no non-trivial characters.   The Cohen-Macaulay property can be handled in a similar manner, as $G^{n-1}$ is smooth, and taking unipotent invariants introduces at worst rational singularities by \cite[Theorem 6]{popovcontractions} (see also \cite[Theorem 1.3]{M14}).

\end{proof}

When $G$ is semisimple its Weyl chamber $\Delta$ is a simplicial cone, this means that it can be identified with a positive orthant in a real vector space.  As a consequence $\mathcal{T}_3(G)$ and $T_3(G)$ are subsets of a positive orthant, and the multigraded Hilbert function $H_{T^3}(P_3(G), \mathbf{t})$ satisfies the assumptions in Theorem \ref{stanley}.

\subsection{Toric degeneration of $P_3(G)$}\label{toricp3}

Now we discuss the toric degenerations of $P_3(G)$ constructed in \cite{M14} in more detail. If $G$ is connected and reductive, then for each choice $\mathbf{i} \in R(w_0)$ of a reduced decomposition of the longest element of the Weyl group $W$ of $G$ there is a flat degeneration of $P_3(G)$ to a normal toric variety $X(C_{3, \mathbf{i}}(G))$.   These degenerations are all equivariant with respect to the action of $T^3$ on $P_3(G)$, so it follows that there is a surjective map of polyhedral cones:

\begin{equation}
\pi_3: \mathcal{C}_{3, \mathbf{i}}(G) \to \mathcal{T}_3(G).
\end{equation}

\noindent
The map $\pi_3$ restricts to a surjective map of affine semigroups $\pi_3: C_{3, \mathbf{i}}(G) \to T_3(G)$. As the toric degenerations constructed in \cite{M14} are $T^3$--equivariant,  they must preserve the multigraded Hilbert function of $P_3(G)$:

\begin{equation}
H_{T^3}(\C[C_{3, \mathbf{i}}(G)], \mathbf{t}) = H_{T^3}(\C[P_3(G)], \mathbf{t}).
\end{equation}

\begin{lemma}
Let $G$ be semisimple and simply connected, then there is a point $\omega_{3, \mathbf{i}}(G) \in C_{3, \mathbf{i}}(G)$ such that $int(C_{3, \mathbf{i}}(G)) = \omega_{3, \mathbf{i}}(G) + C_{3, \mathbf{i}}(G)$, in particular $\C[C_{3, \mathbf{i}}(G)]$ is Gorenstein. Furthermore, for $(\lambda_1, \lambda_2, \lambda_3) = \pi_3(\omega_{3, \mathbf{i}}(G)) \in T_3(G),$ we must have $\lambda_1 = \lambda_2 = \lambda_3 = \lambda_G$ with $\lambda_G^* = \lambda_G.$
\end{lemma}

\begin{proof}
By the Theorem \ref{stanley} $\C[C_{3, \mathbf{i}}(G)]$ must be Gorenstein.  It follows that $\omega_3(G)$ must exist with the property $int(C_{3, \mathbf{i}}(G)) = \omega_{3, \mathbf{i}}(G) + C_{3, \mathbf{i}}(G)$, and by Proposition \ref{symmetricgenerator}, $\pi_3(\omega_{3, \mathbf{i}}(G)) \in T_3(G)$ is preserved by any finite group of $\Z$-linear symmetries of $H_{T^3}(P_3(G), \mathbf{t})$.    

Triple tensor product multiplicities are invariant with respect to permutation of the weights $\lambda_1, \lambda_2, \lambda_3$, this is because tensor product of representations is a commutative operation. Consequently, the three weights $(\lambda_1, \lambda_2, \lambda_3) = \pi_3(\omega_{3, \mathbf{i}}(G))$ must all coincide: $\lambda_1 = \lambda_2 = \lambda_3 = \lambda_G.$  Moreover, the space of invariants in the tensor product $V(\lambda) \otimes V(\eta) \otimes V(\mu)$ is dual to the space of invariants in $V(\lambda^*) \otimes V(\eta^*) \otimes V(\mu^*)$, so $\lambda_G = \lambda_G^*.$
\end{proof}

\begin{example}{\rm
When $G = \SL_m(\C)$ and a particular choice of reduced word $\mathbf{i},$ the semigroup $C_{3, \mathbf{i}}(G)$ is the cone of so-called Berenstein-Zelevinsky triangles $BZ_m$.   A point $x \in BZ_m$ is defined by a triangular array of integers as in Figure \ref{BZtriangle}.  
The entries of $x \in BZ_m$ must all be non-negative integers, and any two pairs $a,b; A, B$ of integers on opposite sides of a hexagon in the array must have the same sum: $a + b = A + B.$  The generator $\omega_{BZ_m}$ is depicted on the right hand side of Figure \ref{BZtriangle}, it is the triangle with all entries equal to $1.$  Each assignment of numbers is dual to a weighted graph supported on the vertices of the triangle, see Figure \ref{BZtriangle}. 

To compute the map $\pi_3: BZ_m \to T_3(\SL_m(\C))$ we choose a clockwise orientation on the diagram.  For $x \in BZ_m,$ let $(a_1, \ldots, a_{2m-2}), (b_1, \ldots, b_{2m-2}),  (c_1, \ldots, c_{2m-2})$ be the entries along the three edges of the diagram, ordered by the chosen orientation.  The three weights $\pi_3(x) \in T_3(\SL_m(\C))$ are then $[(a_1 + a_2, \ldots, a_{2m-3} + a_{2m-2}), (b_1 + b_2, \ldots, b_{2m-3} + b_{2m-2}), (c_1 + c_2, \ldots, c_{2m-3} + c_{2m-2})]$.   Following this recipe, $\pi_3(\omega_3(\SL_m(\C))) = [(2, \ldots, 2), (2,\ldots , 2), (2,\ldots, 2)]$.

\begin{figure}[htbp]
\centering
\includegraphics[scale = 0.5]{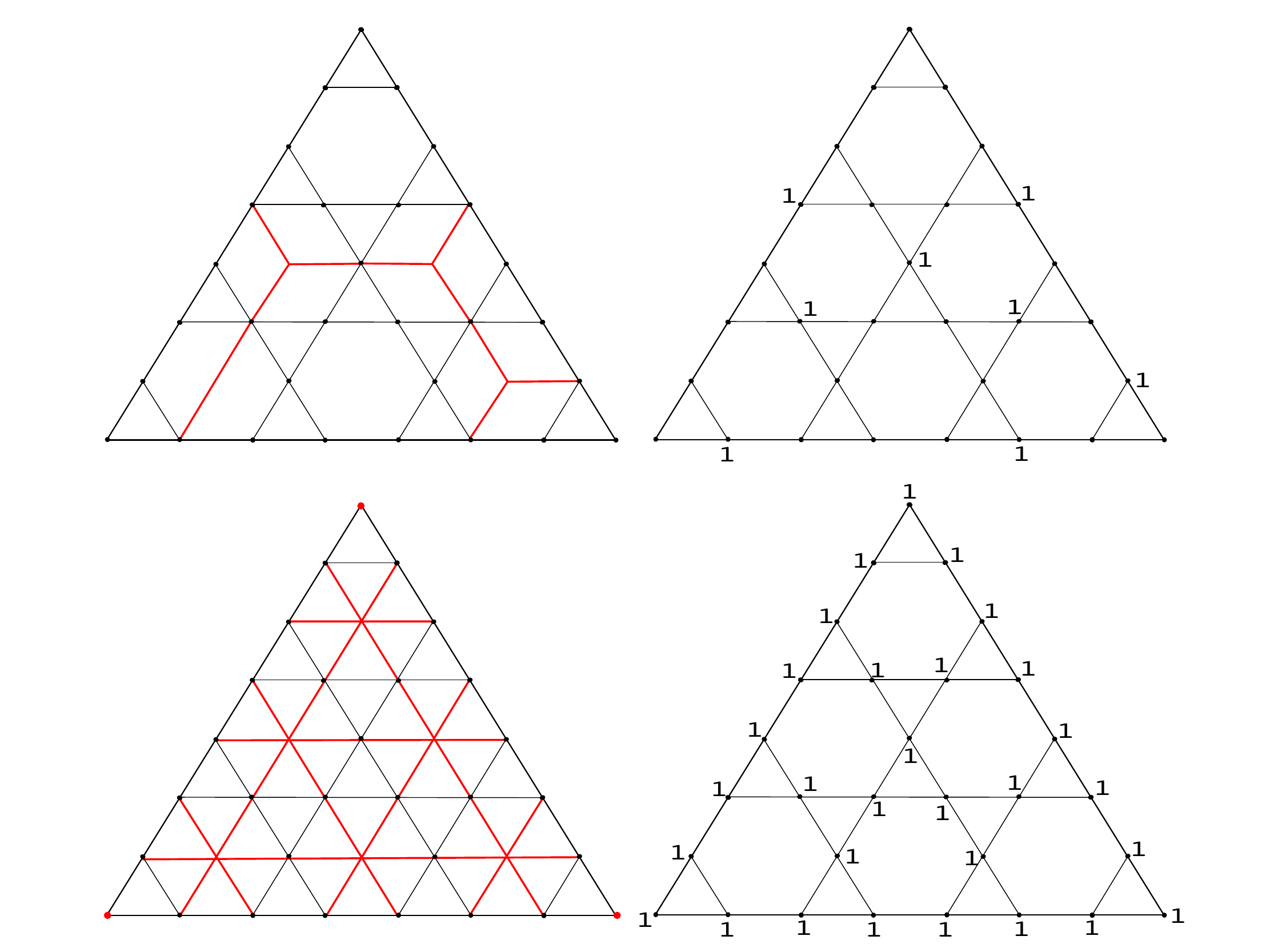}
\caption{Berenstein-Zelevinksy triangles (right) with graphical representations (left).  The element $\omega_3(\SL_5(\C))$ is represented in the bottom row. }
\label{BZtriangle}
\end{figure}
}
\end{example}

Recall that any connected reductive group $G$ is a finite quotient $G = [G^{ss}\times \hat{Z}]/K$, where $\hat{Z}$ is a torus, $G^{ss}$ is semisimple and simply connected, and $K \subset Z(G^{ss}) \times \hat{Z}$.  

\begin{lemma}\label{root}
Let $G$ be reductive, with $G = [G^{ss}\times \hat{Z}]/K$ as above, and let $\lambda \in \Delta$ be in the root lattice $\mathcal{R}(G^{ss})$.  Then $(\lambda, 0)$ is a dominant weight of $G.$
\end{lemma}

\begin{proof}
The weight $(\lambda,0)$ defines a character on $K \subset Z(G^{ss}) \times \hat{Z}.$ The $\hat{Z}$ component of $(\lambda, 0)$ is $0$, so it follows that this character lifts to a character on the image of $K$ in $Z(G^{ss})$.  As $\lambda$ is in the root lattice, it must define a trivial character on $Z(G^{ss})$, and therefore $K.$
\end{proof}

Tensor product invariants of $G$ can be described using the tensor product invariants of $G^{ss}$ and $\hat{Z}$.  The tensor product invariant space $(V(\lambda_1, \rho_1) \otimes V(\lambda_2, \rho_2) \otimes V(\lambda_3, \rho_3))^G$ is isomorphic to $(V(\lambda_1) \otimes V(\lambda_2) \otimes V(\lambda_3))^{G^{ss}}$, and we must have $\rho_1 + \rho_2 + \rho_3 = 0$ as a character of $\hat{Z}.$ 

A choice of reduced word decomposition $\mathbf{i}$ for $G$ imposes a choice for $G^{ss}$. Accordingly, following \cite{M14}, we can describe the polyhedral cone $\mathcal{C}_{3, \mathbf{i}}(G)$ as the product of $\mathcal{C}_{3, \mathbf{i}}(G^{ss})$ with the cone $\mathcal{T}_3(\hat{Z}) = \{(\rho_1, \rho_2, \rho_3) \ |\ \rho_i \in \mathcal{X}(\hat{Z}),  \rho_1 + \rho_2 + \rho_3 = 0\}$.  The affine semigroup $C_{3, \mathbf{i}}(G)$ is then the subsemigroup of the product $C_{3, \mathbf{i}}(G^{ss})$ with $T_3(\hat{Z})$ picked out by the sublattice defined by requiring $\lambda_i + \rho_i$ to be the $0$ character on $K$.   

\begin{theorem}\label{p3gor}
For any reductive group $G$, the affine semigroup algebra $\C[C_{3, \mathbf{i}}(G)]$ and the algebra $\C[P_3(G)]$ are Gorenstein.
\end{theorem}

\begin{proof}
We pass to $G^{ss}$, for $G =[G^{ss}\times \hat{Z}]/K.$  The affine semigroup algebra $\C[C_{3, \mathbf{i}}(G^{ss})]$ is Gorenstein by Lemma \ref{pnufdgor} and Theorem \ref{stanley}, therefore so is the algebra of Laurent polynomials $\C[C_{3, \mathbf{i}}(G^{ss})\times T_3(\hat{Z})] = \C[C_{3, \mathbf{i}}(G^{ss})]\otimes_{\C} \C[T_3(\hat{Z})].$ As $\C[C_{3, \mathbf{i}}(G)]$ is obtained from $\C[C_{3, \mathbf{i}}(G^{ss})\times T_3(\hat{Z})] $ by intersecting the semigroup $C_{3, \mathbf{i}}(G^{ss})\times T_3(\hat{Z})$ with the sublattice described above, to prove the theorem it suffices to show that $(\lambda_{G^{ss}},0)$ is a dominant weight of $G.$  By Lemma \ref{root} this follows if $\lambda_{G^{ss}}$ is in the root lattice of $G^{ss}.$  Since $V(\lambda_{G^ss})^{\otimes 3}$ contains an invariant, and $\lambda_{G^{ss}}$ is self-dual, we must have that $\lambda_{G^{ss}} + \lambda_{G^{ss}} - \lambda_{G^{ss}}^* = \lambda_{G^{ss}}$ is a member of the root lattice. 
\end{proof}

\section{Free group character varieties are Gorenstein}\label{gor-proof-sec}

In Section \ref{application-sec} we showed that the configuration spaces $P_3(G)$ are Gorenstein when $G$ is connected and reductive using
degeneration methods and the combinatorics of affine semigroups; now we use the same methods on $\mathcal{X}(F_g, G)$.  In \cite{M14} it is shown that one can construct a degeneration of $\mathcal{X}(F_g, G)$ to a normal affine toric variety $X(C_{\Gamma, \mathbf{i}}(G))$ which depends on a choice of the following data:

\begin{enumerate}
\item A trivalent graph $\Gamma$ with no leaves and $\beta_1(\Gamma) = g,$
\item An assignment $\mathbf{i}: V(\Gamma) \to R(w_0)$ of reduced word decompositions of $w_0$ to the vertices of $\Gamma.$
\end{enumerate}

\noindent
The information $(\Gamma, \mathbf{i})$ is used to define a convex polyhedral cone $\mathcal{C}_{\Gamma, \mathbf{i}}(G)$ with affine
semigroup $C_{\Gamma, \mathbf{i}}(G).$  By Proposition \ref{deggor}, we can prove that $\mathcal{X}(F_g, G)$ is Gorenstein provided we can establish that $X(C_{\Gamma, \mathbf{i}}(G))$ is Gorenstein. Happily, the cone $\mathcal{C}_{\Gamma, \mathbf{i}}(G)$ and the semigroup $C_{\Gamma, \mathbf{i}}(G)$ are constructed from the cones $\mathcal{C}_{3, \mathbf{i}(v)}(G)$ and semigroups $C_{3, \mathbf{i}(v)}(G)$, $v \in V(\Gamma)$, which are Gorenstein by Theorem \ref{p3gor}. 

Fix a trivalent graph $\Gamma$ as above, and consider the forest $\hat{\Gamma}$ obtained by splitting each edge $e \in E(\Gamma)$ into two edges.  This procedure results in $|V(\Gamma)|$ connected components, each a trinode $\tau_v$ for $v \in V(\Gamma).$ We let $e(v, 1), e(v, 2), e(v, 3)$ be the three edges of the trinode $\tau_v$, and we let $\Psi: \hat{\Gamma} \to \Gamma$ be the graph map which sends two edges $e(v, i), e(w, j)$ to the same edge $e \in E(\Gamma)$ whenever $e(v, i)$ and $e(w, j)$ are obtained from $e$ in the construction of $\hat{\Gamma}.$

\begin{figure}[htbp]
\centering
\includegraphics[scale = 0.3]{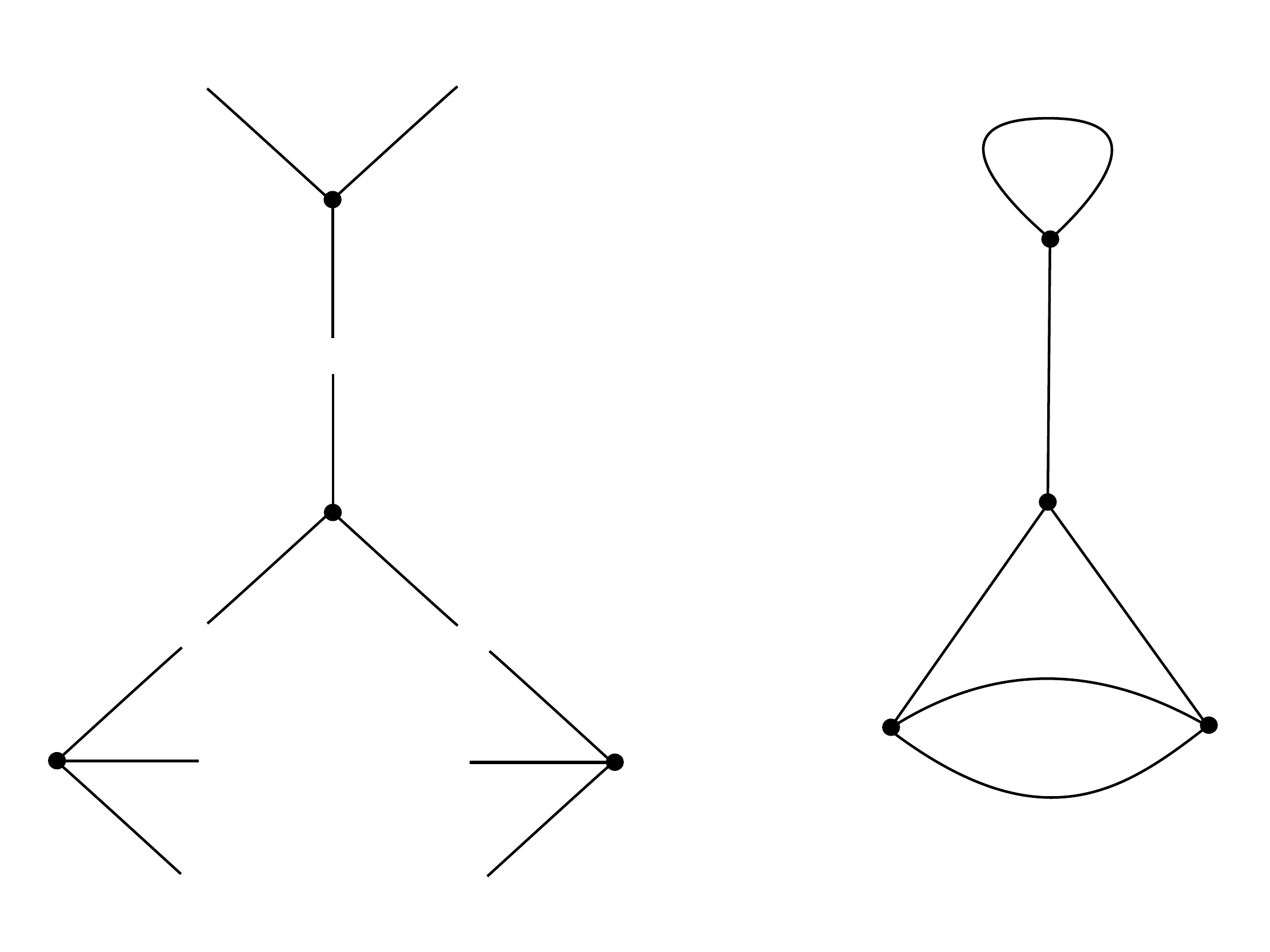}
\caption{A trivalent graph $\Gamma$ (right), with the associated forest $\hat{\Gamma}$ (left).}
\label{Graphs}
\end{figure}

 For each $v \in V(\Gamma)$ we label the components of the projection map $\pi_3: \mathcal{C}_{3, \mathbf{i}(v)}(G) \to \mathcal{T}_3(G)$ so that the image of a point $x \in \mathcal{C}_{3, \mathbf{i}(v)}(G)$ is a triple $(\pi_{e(v,1)}(x), \pi_{e(v,e)}(x), \pi_{e(v,3)}(x)) \in \mathcal{T}_3(G).$  

\begin{definition}
The cone $\mathcal{C}_{\Gamma, \mathbf{i}}(G)$ is the subcone of $\prod_{v \in V(\Gamma)} \mathcal{C}_{3, \mathbf{i}(v)}(G)$ defined by the equations $\pi_{v, i}(x) = \pi_{w, j}(x)^*$ whenever $\Psi(e(v, i)) = \Psi(e(w, j))$.  The semigroup $C_{\Gamma, \mathbf{i}}(G)$ is likewise the intersection $\mathcal{C}_{\Gamma, \mathbf{i}}(G) \cap \prod_{v \in V(\Gamma)} C_{3, \mathbf{i}(v)}(G)$.
\end{definition}

Recall the interior point $\omega_{3, \mathbf{i}}(G) \in C_{\Gamma, \mathbf{i}}(G)$, the next lemma shows that an analogous interior point exists in $C_{\Gamma, \mathbf{i}}(G)$.

\begin{lemma}
The point $\omega_{\Gamma, \mathbf{i}}(G) = \prod_{v \in V(\Gamma)} \omega_{3, \mathbf{i}(v)}(G)$ is in $C_{\Gamma, \mathbf{i}}(G) \subset \mathcal{C}_{\Gamma, \mathbf{i}}(G)$
\end{lemma}

\begin{proof}
Stated in the notation above, the dominant weights $\pi_{e(v, i)}(\omega_{3, \mathbf{i}(v)}(G))$ are all the same dominant weight $\lambda(G)$, irrespective of the choice of vertex $v$ or the decomposition $\mathbf{i}(v)$.  Furthermore, $\lambda(G)$ is self-dual.
\end{proof}

\begin{corollary}\label{gorproof}
Every interior point $x \in int(C_{\Gamma, \mathbf{i}}(G))$ can be written $y + \omega_{\Gamma, \mathbf{i}}(G)$ for some $y \in C_{\Gamma, \mathbf{i}}(G)$. 
\end{corollary}

\begin{proof}
As $int(\mathcal{C}_{\Gamma, \mathbf{i}}(G)) \cap int(\prod_{v \in V(\Gamma)} \mathcal{C}_{3, \mathbf{i}(v)}(G))$ is non-empty, we conclude
that $int(\mathcal{C}_{\Gamma, \mathbf{i}}(G))$ is contained in $int(\prod_{v \in V(\Gamma)} \mathcal{C}_{3, \mathbf{i}(v)}(G))$. It follows
that any $x \in  int(C_{\Gamma, \mathbf{i}}(G))$ can be written as $y + \omega_{\Gamma, \mathbf{i}}(G)$ for $y \in \prod_{v \in V(\Gamma)} \C_{3, \mathbf{i}(v)}(G).$ As $\omega_{\Gamma, \mathbf{i}}(G)$ and $x$ satisfy the linear conditions $\pi_{v, i}(x) = \pi_{w, j}(x)^*$ above, $y$ must satisfy these conditions as well.  We conclude that $y \in C_{\Gamma, \mathbf{i}}(G)$.
\end{proof}

By Corollary \ref{gorproof}, $X(C_{\Gamma, \mathbf{i}}(G))$ is always a Gorenstein toric variety, so $\mathcal{X}(F_g, G)$ must be Gorenstein as well.  This concludes the proof of Theorem \ref{maintheorem}. 

\begin{remark}\label{pngor}{\rm
An identical argument to the one given in this section can be used to prove that $P_n(G)$ is a Gorenstein variety for $G$ connected and reductive, and $n > 3$.  This follows from an analogous use of the toric degeneration from \cite{M14}, whose existence is also proved in \cite{M14}, and replacing the trivalent graph $\Gamma$ with a trivalent tree $\tree$ with $n$ leaves.   }
\end{remark}

\begin{example}{\rm
Continuing with the example $G = \SL_m(\C)$, we may choose $\mathbf{i}: V(\Gamma) \to R(w_0)$ to give the cone $BZ_m$ of Berenstein-Zelevinksy triangles at each vertex of $\Gamma.$   We orient all triangles so that if two triangles $x, y$ meet along edges $e, f$, then these edges have opposite directions.  In this case, $x$ and $y$ can be glued if their respective entries $(a_1, \ldots, a_{2m-2})$ and $(A_1, \ldots, A_{2m-2})$ satisfy $ a_{2i-1} + a_{2i} = A_{2(m-i) -3} + A_{2(m-i) -2}$.  This condition corresponds to requiring the sums of pairs which line up geometrically in the plane to be equal. We call the resulting object $x_{\Gamma} \in C_{\Gamma, \mathbf{i}}(\SL_m(\C)) = BZ_{\Gamma, m}$ a Berenstein-Zelevinksy quilt.  Quilts can be thought of as a kind of combinatorial oriented surface dual to $\Gamma.$  Acyclic pieces of two quilts, with corresponding dual weighted graphs, are depicted in Figure \ref{Quilts}.

\begin{figure}[htbp]
\centering
\includegraphics[scale = 0.5]{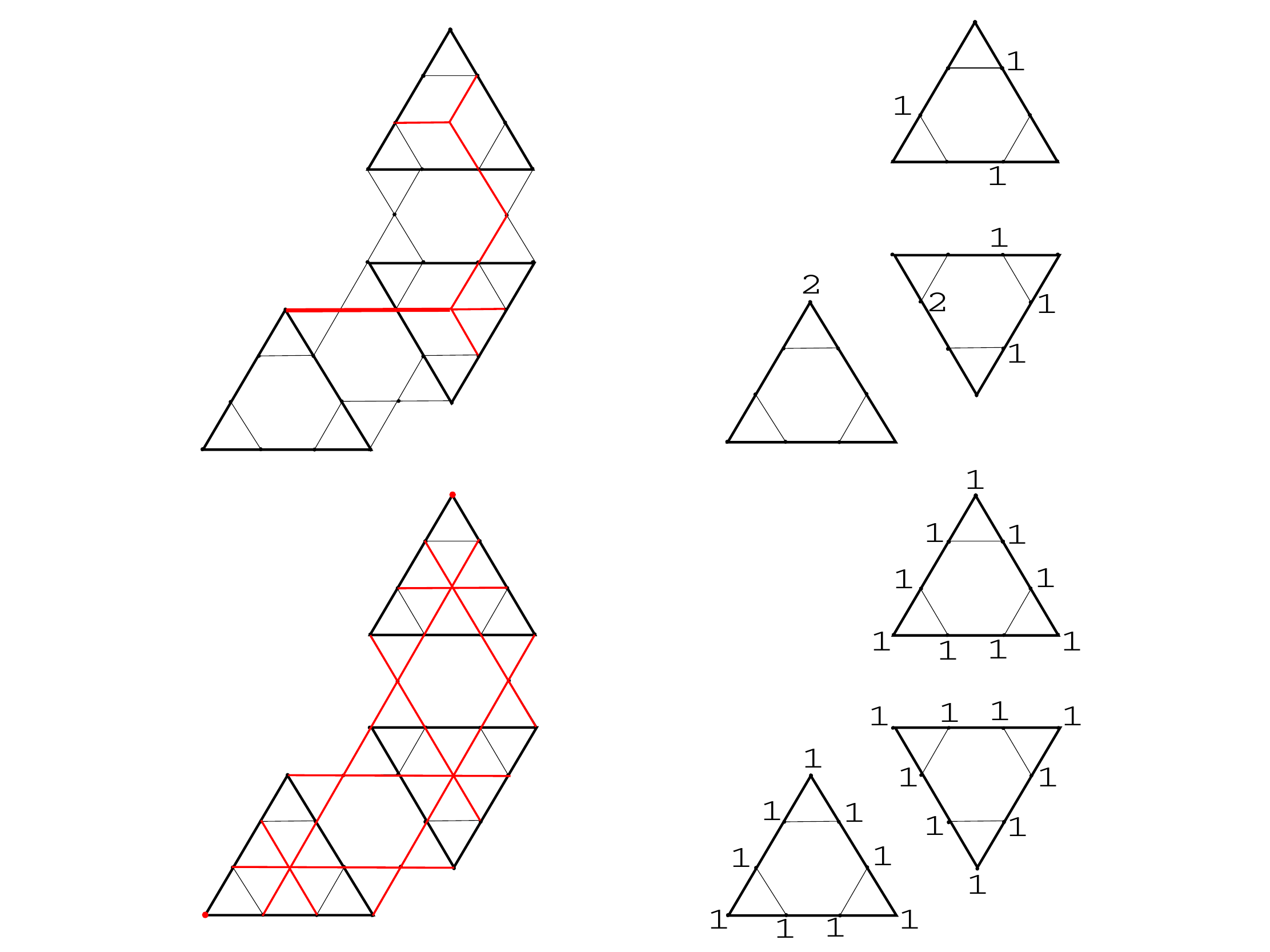}
\caption{Berenstein-Zelevinsky quilts (right) with their graphical representations (left).  The element $\omega_{\Gamma}(\SL_3(\C))$ is represented in the bottom row.}
\label{Quilts}
\end{figure}
}
\end{example}

\newcommand{\etalchar}[1]{$^{#1}$}

\end{document}